\documentclass{amsart}

\usepackage{amssymb}
\usepackage{url}                    

\theoremstyle{plain}
\newtheorem{lem}{Lemma}[section]

\newtheorem{thm}[lem]{Theorem}

\newtheorem{prb}[lem]{Problem}
\newtheorem{prp}[lem]{Proposition}
\theoremstyle{definition}

\newtheorem*{exm}{Example}

\newcommand{\inv}{^{-1}}                
\newcommand{\Id}[1]{\mathrm{id}_{#1}}   
\newcommand{\IS}{\mathcal{I}}           
\newcommand{\Aut}{\mathrm{Aut}}         
\newcommand{\PAut}{\mathrm{PAut}}       
\newcommand{\img}{\mathrm{img}}         
\newcommand{\TC}{\mathrm{TC}}           
\newcommand{\tr}{\mathrm{tr}}           

\DeclareMathOperator{\greenL}{\mathcal{L}}
\DeclareMathOperator{\greenR}{\mathcal{R}}
\DeclareMathOperator{\greenH}{\mathcal{H}}

\title{Abelianness and centrality in inverse semigroups}

\author{Michael Kinyon}
\address[Kinyon]{Department of Mathematics, University of Denver, Denver, CO 80208, USA}
\email{mkinyon@du.edu}

\author{David Stanovsk\'{y}${}^\dag$}
\thanks{${}^\dag$ Partially supported by the cooperation grant LTAUSA19070 from M\v{S}MT \v{C}R}
\address[Stanovsk\'{y}]{Department of Algebra, Faculty of Mathematics and Physics,
Charles University, Sokolovsk\'{a} 83, 18675 Praha 8, Czech Republic}
\email{stanovsk@karlin.mff.cuni.cz}

\date{\today}

\begin{document}

\begin{abstract}
We adapt the abstract concepts of abelianness and centrality of universal algebra to the
context of inverse semigroups. We characterize abelian and central congruences in terms of
the corresponding congruence pairs. We relate centrality to conjugation in inverse semigroups.
Subsequently we prove that solvable and nilpotent inverse semigroups are groups.
\end{abstract}



\maketitle

\section{Introduction}

\subsection{Motivation}

Abstract commutator theory of universal algebra \cite{FM} was developed to capture, among other things, the concepts of abelianness and centrality, and consequently,
solvability and nilpotence for general algebraic structures. Commutator theory proved immensely useful in solving various problems of universal algebra, equational logic
and combinatorics of functions; see \cite{MS} for a list of applications and references.

Interpreting the abstract concepts of commutator theory in a concrete class is, in general, not an easy task. In groups, it is not difficult to show that the universal algebraic notions coincide with the classical ones. In less familiar structures like loops (quasigroups with identity elements), this is not the case: universal algebraic nilpotence coincides with the classical central nilpotence, but classical and universal algebraic solvability are different \cite{SV,SV2}. Another example of a successful adaptation of commutator theory is the class of racks and quandles \cite{BS}.

The ideas surrounding commutator theory, such as term conditions, were first adapted to semigroup theory by McKenzie \cite{McKsemi}, Taylor \cite{Taylor} and Warne \cite{Warne}. The latter showed, for example, that an abelian regular semigroup is a subsemigroup of a rectangular group with commutative group component. We note that although in semigroup theory, the word ``abelian'' is often used as a synonym for ``commutative'' (as in group theory), in this paper we will only use ``abelian'' in the universal algebra sense. Commutator theory itself has recently been applied to Rees matrix semigroups \cite{MR}, inverse semigroups \cite{RM1} and regular semigroups \cite{RM2}.

In the present paper, we focus on inverse semigroups \cite{Lawson,Petrich}. Inverse semigroups include groups and semilattices as special cases, and as advocated in \cite{Lawson}, they are the natural algebraic framework for the study of partial symmetries. They have applications in a variety of areas, such as logic and $\lambda$-calculus, operator algebras \cite{Paterson}, automata theory and computer science. It is therefore rather natural to see what exactly commutator theory can bring to inverse semigroups.

Commutator theory has strongest properties for congruence modular varieties; for instance, this assumption is used to prove commutativity of the commutator. Inverse semigroups possess a weak difference term \cite{Lipp}, but do not form a congruence modular variety, hence one cannot expect the commutator to shine in its full strength. Therefore we specialize to the derived notions of abelianness and centrality. We characterize abelian and central congruences (Theorem \ref{Thm:abelian and central congruence}), we describe the center of an inverse semigroup (Theorem \ref{thm:center}), and subsequently characterize solvable / nilpotent inverse semigroups as solvable / nilpotent groups (Theorem \ref{Thm:solvnilp}).

Our three main theorems are formulated more precisely in the next subsection, and their proof occupies most of the paper. Alternative approaches to nilpotence and solvability in inverse semigroups can be found in \cite{Kowol_Mitsch,Malcev,Meldrum,Piochi}; we will discuss these in Section \ref{s:others}.

\subsection{Main results}

Our first theorem is a characterization of abelian and central congruences in terms of their corresponding congruence pairs.

\begin{thm}\label{Thm:abelian and central congruence}
Let $S$ be an inverse semigroup. Then
\begin{enumerate}
	\item abelian congruences of $S$ correspond to congruence pairs $(N,0_{E(S)})$ such that $N$ is a commutative normal inverse subsemigroup of $S$.
	\item central congruences of $S$ correspond to congruence pairs $(N,0_{E(S)})$ such that $N$ is a normal inverse subsemigroup of $S$ contained in
\[
Z(S) = \{a\in S\mid axa\inv a = aa\inv x a,\ \forall x\in S\}\,.
\]
\end{enumerate}
\end{thm}

Consequently, $(Z(S),0_{E(S)})$ corresponds to the largest central congruence, called (in universal algebra) the \emph{center}.

\begin{exm}
If $S$ is a group, then $Z(S)$ is the classical group-theoretic center. If $S$ is a semilattice, then the center is trivial, i.e., $Z(S)=E(S)$. The center of the symmetric inverse semigroup $\IS(X)$ of partial bijections on a set $X$ is also trivial.
\end{exm}

There is a remarkable analogy to the group theoretic center. The center of a group can be described in terms of conjugation.
Indeed, the center of a group $G$, although usually just defined as the set of elements that commute with all elements, turns out to be precisely the kernel of the homomorphism $G\to\Aut(G);\quad g\mapsto\phi_g$, where $\phi_g$ is the inner automorphism $x\mapsto gxg\inv$.

In a semigroup $S$, the word ``center'' usually refers to the subsemigroup $C(S) = \{ a\in S\mid ax=xa\,\forall x\in S\}$; see Section \ref{s:others} for how this is related to our $Z(S)$. A better analogy to the group-theoretic center is based on conjugation. As an equivalence relation, conjugation in inverse semigroups is thoroughly studied in \cite{AKKM}. Here we think of it as defining mappings on inverse semigroups. In fact, conjugation by a fixed element $g$ defines two closely related maps. First, it defines the total transformation $\psi_g : S\to S;\quad x\mapsto gxg\inv$. Second, and more closely analogous to the group case, is the partial mapping $\phi_g : g\inv Sg\to gSg\inv$, which turns out to be a partial automorphism, i.e., an isomorphism of inverse subsemigroups. As we will see, the center of an inverse semigroup $S$ is both the kernel (relation) of the homomorphism
\[
\Psi : S\to T_S;\quad g\mapsto \psi_g
\]
into the semigroup $T_S$ of all total transformations of $S$, and also the kernel of the homomorphism
\[
\Phi : S\to \mathcal\PAut(S);\quad g\mapsto \phi_g
\]
into the inverse semigroup $\PAut(S)$ of partial automorphisms of $S$. (Details are in Section \ref{sec:inner_aut}.)

Preston \cite{Preston} was the first to study the homomorphism $\Phi : S\to \PAut(S)$ and to describe its corresponding congruence pair. Mel'nik \cite{Melnik} and Petrich \cite{PetrichHull} later proved similar results independently, the latter as part of his larger theory of conjugate extensions. Petrich's paper does not cite Preston's, but in his later book \cite{Petrich}, Petrich credits both Preston and Mel'nik. Petrich called our $Z(S)$ the \emph{metacenter} of $S$; Mel'nik had earlier called it the \emph{quasicenter}.

The center in our (universal algebra) sense can also be characterized syntactically, in analogy with the standard group-theoretical description. All descriptions are summarized in the following theorem.

\begin{thm}\label{thm:center}
Let $S$ be an inverse semigroup. The following congruences are equal:
\begin{enumerate}
	\item the center of $S$ in the sense of \cite{FM},
	\item $\ker(\Psi)$,
    \item $\ker(\Phi)$,
	\item $\greenH\,\cap\, \{(a,b)\in S\times S \mid axb = bxa,\ \forall x\in S\}$.
\end{enumerate}
The corresponding congruence pair is $(Z(S),0_{E(S)})$.
\end{thm}

As a consequence of our description of abelian and central congruences we characterize solvable and nilpotent (in the sense of univeral algebra) inverse semigroups in Section \ref{s:nilpotence}. Sadly, they are groups.

\begin{thm}\label{Thm:solvnilp}
An inverse semigroup is nilpotent (resp. solvable) if and only if it is a nilpotent (resp. solvable) group.
\end{thm}

Theorem \ref{Thm:solvnilp} was independently found by Radovi\'{c} and Mudrinski \cite{RM1} using a completely different methodology.

\section{Facts about inverse semigroups}

\subsection{Basic facts}

In the present section, we summarize basic facts about inverse semigroups we need in out proofs. For details, we refer to the textbooks \cite{Howie,Petrich}.

A semigroup $S$ is \emph{regular} if each element $x\in S$ has an \emph{inverse} $x'\in S$ satisfying $xx'x=x$ and $x'xx'=x'$. In case
each element $x$ has a unique inverse $x\inv$, $S$ is said to be an \emph{inverse semigroup}. Equivalently, a semigroup $S$ is inverse if and only if it is regular and the set $E(S) = \{ e\in S\mid s^2 = s\}$ of idempotents is a commutative subsemigroup, i.e. a semilattice.

As is customary, we will view an inverse semigroup $S$ as an algebraic structure $(S,\cdot,{}\inv)$ with both a binary operation and a unary operation. As such, inverse semigroups also satisfy the identities $(x\inv)\inv = x$ and $(xy)\inv = y\inv x\inv$. Subsemigroups of inverse semigroups which are closed under inversion $x\mapsto x\inv$ are called \emph{inverse subsemigroups}.

In semigroups, Green's relation $\greenL$ is defined by $a\,\greenL\,b$ if and only if $a=b$ or $a=ub$ and $b=va$ for some $u,v\in S$. Green's relation $\greenR$ is defined dually, and Green's relation $\greenH = \greenL\cap \greenR$. The relation $\greenL$ is a right congruence, that is, $a\,\greenL\,b$ implies $ax\,\greenL\,bx$ for all $x$. Dually, $\greenR$ is a left congruence. In an inverse semigroup, the $\greenL$-class containing an element $a$ has a unique idempotent, namely $a\inv a$. Similarly, the $\greenR$-class containing $a$ has a unique idempotent, namely $aa\inv$. In particular, $\greenL$ and $\greenR$ can be determined equationally: $a\,\greenL\,b$ if and only if $a\inv a=b\inv b$ and $a\,\greenR\,b$ if and only if $aa\inv = bb\inv$. Another consequence is that an inverse semigroup $S$ is a group if and only if $|E(S)|=1$.

\subsection{Congruences}

Let $S$ be an inverse semigroup. An equivalence $\alpha$ on $S$ is called a \emph{congruence} of $S$ if it is invariant with respect to multiplication and inversion, that is, it is an inverse subsemigroup of $S\times S$. For any homomorphism $\phi : S\to T$ of inverse semigroups, $\ker(\phi) = \{(a,b)\in S\times S\mid \phi(a)=\phi(b)\}$ is a congruence.

The following is a version of Lallement's Lemma (\cite{Howie}, Lem.~2.2.4), which will be suitable for our purposes.

\begin{lem}\label{Lem:homom_images}
Let $S$ be an inverse semigroup and let $\phi : S\to T$ be a semigroup homomorphism into a semigroup $T$. Then $\img(\phi)$ is an inverse semigroup and $\phi$ preserves inverses. Consequently, $\ker(\phi)$ is a congruence.
\end{lem}
\begin{proof}
For every $a\in S$, we have $\phi(a)\phi(a\inv)\phi(a)=\phi(aa\inv a)=\phi(a)$; similarly, $\phi(a\inv)\phi(a)\phi(a\inv)= \phi(a\inv)$.
Thus $\img(\phi)$ is regular. If $\phi(a)\in \img(\phi)$ is an idempotent, set $h = aa^{-2}a$. Then $hh=aa^{-2}a^2a^{-2}a = aa^{-2}a=h$ and $\phi(h) = \phi(a)\phi(a^{-2})\phi(a) = \phi(a)^2 \phi(a^{-2})\phi(a)^2 = \phi(a^2 a^{-2} a^2) = \phi(a^2) = \phi(a)^2 = \phi(a)$. Thus every idempotent in $\img(\phi)$ is the image of an idempotent in $S$. For $e,f\in E(S)$, $\phi(e)\phi(f) = \phi(ef) = \phi(fe) = \phi(f)\phi(e)$, which shows that idempotents in $\img(\phi)$ commute. Therefore $\img(\phi)$ is an inverse semigroup. Since $\phi(a\inv)$ is an inverse of $\phi(a)$, it must be the unique inverse, that is, $\phi(a)\inv = \phi(a\inv)$ for all $a\in S$. For the remaining assertion, it is enough to show that $(a,b)\in\ker(\phi)$ implies $(a\inv,b\inv)\in \ker(\phi)$. This follows from $\phi(a\inv)=\phi(a)\inv = \phi(b)\inv = \phi(b\inv)$.
\end{proof}

An inverse subsemigroup $N$ of an inverse semigroup $S$ is said to be \emph{normal} if it is stable under conjugation ($g\inv Ng \subseteq N$ for all $g\in S$) and full ($E(S)\subseteq N$).

Congruences of an inverse semigroup $S$ are in one-to-one correspondence with \emph{congruence pairs} $(N,\varepsilon)$ where $N$ is a normal inverse subsemigroup of $S$, $\varepsilon$ is a congruence of $E(S)$ and
\begin{itemize}
	\item[(CP1)] for every $a\in S$ and $e\in E(S)$, if $ae\in N$ and $e\,\varepsilon\,a\inv a$ then $a\in N$;
	\item[(CP2)] for every $a\in N$ and $e\in E(S)$, $a\inv ea\,\varepsilon\,a\inv ae$.
\end{itemize}
For a congruence $\alpha$ of $S$, the corresponding pair is $(N_{\alpha},\tr(\alpha))$ where
$N_\alpha$ is the \emph{kernel} of $\alpha$, defined as the union of all those $\alpha$-classes which contain idempotents;
and $\tr(\alpha)$ is the \emph{trace} of $\alpha$, defined as $\alpha|_{E(S)}$, the restriction of $\alpha$ to $E(S)$.

In the other direction, for a congruence pair $(N,\varepsilon)$, the corresponding congruence $\alpha$ is given by $a\,\alpha\,b$ if and only if
$ab\inv \in N$ and $aa\inv\,\varepsilon\,bb\inv$.

The smallest and the largest congruences $0_S=\Id S$, $1_S=S\times S$ correspond to congruence pairs $(E(S),0_{E(S)})$ and $(S,1_{E(S)})$, respectively.
In case $\tr(\alpha) = 0_{E(S)}$, $\alpha$ is said to be \emph{idempotent-separating}.

\begin{lem}[\cite{Mitsch}, Prop.~2]\label{Lem:Mitsch}
Let $N$ be a normal inverse subsemigroup of an inverse semigroup $S$. Then $(N,0_{E(S)})$ is a congruence pair if and only if every element of $N$ commutes with every element of $E(S)$.
\end{lem}
\begin{proof}
For idempotent-separating congruences, condition (CP1) is always satisfied, and condition (CP2) says that $a\inv ea=a\inv ae$ for every $e\in E(S)$ and $a\in N$. This is certainly true if every element of $N$ commutes with every $e\in E(S)$. Conversely, multiplying by $a$ on the left and using commutativity of $E(S)$, we obtain $ae=a\underline{a\inv ae}=aa\inv ea=eaa\inv a=ea$.
\end{proof}

\begin{lem}\label{Lem:below H iff tr=0}
Let $\alpha$ be a congruence on an inverse semigroup $S$. Then $\alpha$ is idempotent-separating if and only if $\alpha\subseteq\greenH$.
\end{lem}
\begin{proof}
($\Leftarrow$) If $a\,\alpha\,b$, then $aa\inv \,\alpha\,bb\inv$ and $a\inv a\,\alpha\,b\inv b$. Since $\alpha$ is idempotent-separating,
$aa\inv = bb\inv$ and $a\inv a = b\inv b$, which is equivalent to $a\,\greenH\,b$.
($\Rightarrow$) The restriction of $\greenH$ to $E(S)$ is $0_{E(S)}$.
\end{proof}

\section{Partial automorphisms of normal inverse subsemigroups}\label{sec:inner_aut}

Let $S$ be an inverse semigroup. A partial bijection $\phi : X\to Y$, $X,Y\subseteq S$, is said to be a
\emph{partial automorphism} if $X$ and $Y$ are inverse subsemigroups of $S$ and $\phi$ is an isomorphism from $X$ onto $Y$.
The set $\PAut(S)$ of all partial automorphisms of $S$ is an inverse submonoid of $\IS(S)$, which we call
the \emph{partial automorphism monoid} of $S$.

Throughout this subsection, let $N$ be a normal inverse subsemigroup of an inverse semigroup $S$. Let $T_N$ denote the monoid of all (total) transformations of $N$. For each $g\in S$, define $\psi_g : N\to N$ by $\psi_g(a) = gag\inv$ for all $a\in N$. Further, we define
\begin{equation}\label{Eqn:Psi}
\Psi_N : S\to T_N;\quad g\mapsto \psi_g\,. \tag{$\Psi_N$}
\end{equation}
Because of the identity $(xy)\inv = y\inv x\inv$ and Lemma \ref{Lem:homom_images}, the following holds.

\begin{lem}\label{Lem:psi}
    For each $g\in S$, $\Psi_N : S\to T_N$ is a homomorphism and $\ker(\Psi_N)$ is a congruence.
\end{lem}

\begin{lem}\label{Lem:domain_char}
    For all $g\in S$, $gNg\inv = gg\inv N gg\inv$ is an inverse submonoid of $N$.
\end{lem}
\begin{proof}
  We have $gNg\inv = gg\inv gNg\inv gg\inv \subseteq gg\inv N gg\inv\subseteq gNg\inv$, using normality of $N$. This
  establishes the desired equality. It is immediate that $gNg\inv$ is a subsemigroup of $N$, $gg\inv$ is its unit element, and since
  $(gag\inv)\inv = ga\inv g\inv$, $gNg\inv$ is also closed under inversion.
\end{proof}

\begin{lem}\label{Lem:H_char}
    For all $g,h\in S$,
    \begin{enumerate}
      \item $g\greenR h$ if and only if $gNg\inv = hNh\inv$,
      \item $g\greenL h$ if and only if $g\inv Ng = h\inv Nh$.
    \end{enumerate}
\end{lem}
\begin{proof}
  (1) If $g\greenR h$, then $gg\inv = hh\inv$ and so we have $gNg\inv = gg\inv Ngg\inv = hh\inv Nhh\inv = hNh\inv$, using Lemma \ref{Lem:domain_char}. Conversely, if $gNg\inv = hNh\inv$, then $gg\inv = g\cdot g\inv g\cdot g\inv = hbh\inv$ for some $b\in N$ since $E(S)\subseteq N$. Thus $hh\inv\cdot gg\inv = hh\inv\cdot hbh\inv = hbh\inv = gg\inv$. Reversing the roles of $g$ and $h$, we obtain $gg\inv\cdot hh\inv = hh\inv$. Since idempotents commute, $gg\inv = hh\inv$, that is, $g\greenR h$.

  (2) We have $g\greenL h$ if and only if $g\inv \greenR h\inv$ if and only if $g\inv Ng = h\inv Nh$ by part (1).
\end{proof}

For each $g\in S$, we define partial mappings $\phi_g : g\inv Ng\to gNg\inv\,;\ x\mapsto gxg\inv$.

\begin{lem}\label{lem:partial_aut}
  For each $g\in S$, $\phi_g$ is a partial automorphism of $N$, specifically, it is an isomorphism of $g\inv Ng$ onto $gNg\inv$.
\end{lem}
\begin{proof}
For all $a\in N$, $\phi_{g\inv}\phi_g(g\inv ag) = g\inv gg\inv agg\inv g = g\inv ag$ and similarly $\phi_g\phi_{g\inv}(gag\inv) = gag\inv$. Thus $\phi_g$ is a bijection from $g\inv Ng$ to $gNg\inv$ with inverse bijection $\phi_{g\inv}$. For all $a,b\in S$, we have
\begin{align*}
\phi_g(g\inv ag\cdot g\inv bg) &= gg\inv a \underline{gg\inv} bgg\inv \\
& = gg\inv agg\inv \cdot gg\inv bgg\inv
= \phi_g(g\inv ag)\phi_g(g\inv bg)\,,
\end{align*}
hence $\phi_g$ is a partial automorphism.
\end{proof}

We define
\begin{equation}\label{Eqn:Phi}
\Phi_N : S\to \PAut(N);\quad g\mapsto \phi_g\,. \tag{$\Phi_N$}
\end{equation}

\begin{lem}\label{Lem:compose}
$\Phi : S\to \PAut(N)$ is a homomorphism and $\ker(\Phi_N)$ is a congruence.
\end{lem}
\begin{proof}
We first show that the domains of $\phi_g\phi_h$ and $\phi_{gh}$ coincide. The domain of $\phi_{g}\phi_{h}$ is
\[
X = h\inv Nh \cap \{ a\in N\mid hah\inv \in g\inv Ng \}
\]
and the domain of $\phi_{gh}$ is $Y = (gh)\inv Ngh$.

If $h\inv ah\in X$, then
\[
h\inv ah = h\inv (hh\inv ahh\inv)h\in h\inv (g\inv Ng)h = (gh)\inv Ngh = Y\,;
\]
thus $X\subseteq Y$.
For the other inclusion, $Y=(gh)\inv Ngh \subseteq h\inv Nh$, and for all $x\in N$, we have
\begin{align*}
h((gh)\inv xgh)h\inv &= \underline{hh\inv}\underline{g\inv g}g\inv xg\underline{g\inv g}\underline{hh\inv} \\
&= g\inv g hh\inv g\inv xghh\inv g\inv g\in g\inv Ng
\end{align*}
since idempotents commute. Thus $Y\subseteq X$.

Finally, it is clear that $\phi_g\phi_h(a) = \phi_{g h}(a)$ for all $a\in (gh)\inv Ngh$.
\end{proof}

\begin{thm}\label{Thm:kernels}
Let $S$ be an inverse semigroup, let $N$ be a normal inverse subsemigroup of $S$, and let $\Psi_N : S\to T_N$ and $\Phi_N : S\to \PAut(N)$ be defined by \eqref{Eqn:Psi} and \eqref{Eqn:Phi}. Then $\ker(\Psi_N) = \ker(\Phi_N)$. This congruence $\zeta_N$ is idempotent-separating, and the corresponding congruence pair is $(Z_S(N),0_{E(S)})$ where
\[
Z_S(N) = \{ g\in S\mid gg\inv ag = gag\inv g\,, \forall a\in N\}\,.
\]
\end{thm}
\begin{proof}
Assume $(g,h)\in \ker(\Psi_N)$. Since $(g\inv,h\inv)\in \ker(\Psi_N)$, we have $\psi_{g\inv} = \psi_{h\inv}$. In particular, $g\inv Ng = \psi_{g\inv}(N) = \psi_{h\inv}(N) = h\inv Nh$. Thus $\phi_g$ and $\phi_h$ have the same domains. Clearly $\phi_g(x) = \phi_h(x)$ for all $x\in g\inv Ng$, and so $(g,h)\in \ker(\Phi_N)$.

Conversely, assume $(g,h)\in \ker(\Phi_N)$. Since the domains and images of $\phi_g$ and $\phi_h$ coincide, $g\inv Ng = h\inv Nh$ and $gNg\inv = hNh\inv$. By Lemma \ref{Lem:H_char}, $g\greenH h$. For $a\in N$, we have
\begin{align*}
h\inv ah &= h\inv \underline{hh\inv} a\underline{hh\inv} h = h\inv gg\inv agg\inv h = h\inv \phi_g(g\inv ag) h \\
&= h\inv \phi_h(g\inv ag) h = \underline{h\inv h}(g\inv ag)\underline{h\inv h} = g\inv gg\inv agg\inv g = g\inv ag\,.
\end{align*}
Thus $(g\inv,h\inv)\in \ker(\Psi_N)$, and therefore $(g,h)\in \ker(\Psi_N)$.

The preceding paragraph showed $\zeta_N\subseteq \greenH$, hence $\zeta_N$ is idempotent-separating by Lemma \ref{Lem:below H iff tr=0}.

Next we verify that $(Z_S(N),0_{E(S)})$ is a congruence pair.
For all $g\in Z_S(N)$, $e\in E(S)$, $ge = g\underline{g\inv g e} = geg\inv g = \underline{gg\inv e} g = egg\inv g = eg$, using $g\in Z_S(N)$ in the third equality. Thus $g$ commutes with every idempotent. Therefore $(Z_S(N),0_{E(S)})$ is a congruence pair by Lemma \ref{Lem:Mitsch}.

Suppose $g\in S$, $e\in E(S)$ and $g\,\zeta_N\,e$. Since $\zeta_N\subseteq \greenH$ as shown above, $g\inv g = gg\inv = e$. Then for all $a\in N$, $\underline{gag\inv} g = eaeg = eag = gg\inv ag$. Thus $g\in Z_S(N)$.

Conversely, suppose $(g,h)$ is in the congruence corresponding to $(Z_S(N),0_{E(S)})$. Then $gh\inv\in Z_S(N)$ and $g\greenH h$.
Consequently, $gg\inv = hh\inv$, $g\inv g = h\inv h$, and also $(gh\inv)\inv gh\inv = hg\inv gh\inv = hh\inv hh\inv = hh\inv = gg\inv$.  Now, for $a\in N$,
\begin{align*}
  gag\inv &= gg\inv\cdot gag\inv\cdot gg\inv && \\
  &= (gh\inv)\inv \underline{gh\inv\cdot gag\inv\cdot (gh\inv)\inv gh\inv} && \\
  &= \underline{(gh\inv)\inv gh\inv (gh\inv)\inv}\cdot gag\inv\cdot gh\inv && \text{since }gag\inv\in N, gh\inv\in Z_S(N)\,, \\
  &= h\underline{g\inv g}a\underline{g\inv g}h\inv && \\
  &= \underline{hh\inv h}a\underline{h\inv hh\inv} && \text{since }g\,\greenH\,h\\
  &= hah\inv\,.
\end{align*}
Thus $(g,h)\in \ker(\Psi_N) = \zeta_N$. This verifies that $(Z_S(N),0_{E(S)})$ is precisely the congruence pair of $\zeta_N$ and completes the proof.
\end{proof}

The case $N = E(S)$ is well known; $\zeta_{E(S)}$ is the maximal idempotent-separating congruence (\cite{Howie}, pp. 160--161). At the other extreme, we now have the following.

\begin{proof}[Proof of Theorem \ref{thm:center}, (2) coincides with (3); description of congruence pair]$ $\newline
  Apply Theorem \ref{Thm:kernels} to the case $S = N$, where $Z(S) = Z_S(S)$.
\end{proof}


\section{Syntactic description of the center}
\label{s:syntax}

Let $\xi = \{(a,b)\in S\times S \mid axb = bxa,\ \forall x\in S\}$  and $\zeta_S = \ker(\Psi_S) = \ker(\Phi_S)$.

We note that we cannot extend Theorem \ref{Thm:kernels} to include a generalization of item (4) of Theorem \ref{thm:center} to arbitrary normal inverse subsemigroups. Indeed, let $S$ be any noncommutative inverse semigroup and let $N=E(S)$. Then $\zeta_N = S\times S$, but for ever pair $(g,h)$ of noncommuting elements, $g\cdot g\inv g\cdot h = gh\neq hg = h\cdot gg\inv\cdot g$.

\begin{proof}[Proof of Theorem \ref{thm:center}, (2) is contained in (4)]
Suppose $(g,h)\in \zeta_S$. Because $\zeta_S$ is idempotent-separating (Theorem \ref{Thm:kernels}), $g\greenH h$ (Lemma \ref{Lem:below H iff tr=0}), that is, $gg\inv = hh\inv$ and $g\inv g = h\inv h$. It remains to prove that $g\,\xi\, h$. For all $x\in S$, we have
  \begin{align*}
  gxh &= gx\underline{hh\inv}h \\
  &= g\cdot xg\cdot g\inv h \\
  &= \psi_g(xg)h \\
  &= \psi_h(xg)h \\
  &= hxg\underline{h\inv h} \\
  &= hxgg\inv g \\
  &= hxg\,.
  \end{align*}
  We have shown $\zeta_S\subseteq \greenH\cap\,\xi$.
\end{proof}

\begin{proof}[Proof of Theorem \ref{thm:center}, (4) is contained in (2)]
  Suppose $(g,h)\in\greenH\cap\,\xi$. Since $g\greenH h$, we have $gg\inv = hh\inv$ and $g\inv g=h\inv h$.
  Since $gyh = hyg$ for all $y\in S$, we have, for every $x\in S$,
  \begin{align*}
  \psi_g(x) &= gxg\inv = gx\underline{g\inv g}g\inv = \underline{g\cdot xh\inv\cdot h}g\inv \\
  &= hxh\inv \underline{gg\inv} = hxh\inv hh\inv = hxh\inv = \psi_h(x)\,.
  \end{align*}
  We have shown $\greenH \cap\,\xi\subseteq \zeta_S$.
\end{proof}

What remains for Theorem \ref{thm:center} is to prove the equivalence of (1) with the rest.

\section{Centralizing congruences}\label{s:centralizing}

\subsection{Terms}

For all universal algebraic purposes, inverse semigroups are considered in the language $\{\cdot,\,\inv\}$.
(Note that the universal algebraic concepts, such as congruences, or the term condition introduced below, are sensitive to the choice of language. For example, it is fairly difficult to characterize abelian semigroups with respect to the language $\{\cdot\}$ \cite{Warne}, while in our setting, abelian inverse semigroups are just abelian groups).

In inverse semigroups, thanks to associativity and the identities $(xy)\inv = y\inv x\inv$ and $(x\inv)\inv=x$,
every term $t(x_1,\dots,x_n)$ is equivalent to a term of the form $x_{i_1}^{\epsilon_1}x_{i_2}^{\epsilon_2}\cdots x_{i_k}^{\epsilon_k}$,
where $i_1,\dots,i_k\in\{1,...,n\}$ and $\epsilon_1,\dots,\epsilon_k\in\{\pm1\}$, to be called a \emph{normal form} of $t$.

\subsection{The term condition}
Let $A$ be an algebraic structure with congruences $\alpha,\beta$. The \emph{term condition} for an $(n+1)$-ary term $t$ with respect to $\alpha,\beta$, shortly $\TC(t,\alpha,\beta)$, is the following condition:
for every $a\,\alpha\,b$ and every $u_1\,\beta\,v_1$, $\dots$, $u_n\,\beta\,v_n$,
\[  t(a,u_1,\dots,u_n) = t(a,v_1,\dots,v_n)\quad\text{implies}\quad t(b,u_1,\dots,u_n) = t(b,v_1,\dots,v_n).\]
We say that \emph{$\alpha$ centralizes $\beta$}, if the term condition $\TC(t,\alpha,\beta)$ holds for every term $t$ (of any arity $\geq2$).
It is easy to show that $\alpha$ centralizes $\beta$ if and only if $TC(t,\alpha,\beta)$ is satisfied for every term $t$ in which the first variable occurs only once: indeed, we can use the term condition several times to replace every occurrence one-by-one (see \cite[Lemma 4.1]{SV} for a formal proof).
A congruence $\alpha$ is called
\begin{itemize}
\item \emph{abelian} if $\alpha$ centralizes $\alpha$,
\item \emph{central} if $\alpha$ centralizes $1_A$.
\end{itemize}
Subsequently, the \emph{center} of $A$, denoted by $\zeta(A)$, is the largest congruence of $A$ which centralizes $1_A$ (it is not obvious that such a congruence always exists; for inverse semigroups, we show that in Lemmas \ref{Lem:center1} and \ref{Lem:center2}). Hence $\alpha$ is central if and only if $\alpha\leq\zeta(A)$.

In what follows, we will frequently use three special terms:
\[
{\ell}(x,y)=xy,\quad m(x,y,z)=yxz,\quad r(x,y)=yx\,.
\]
The principle observation is that in inverse semigroups, centralization reduces to the term condition for a single term.

\begin{prp}\label{Prp:just_m}
Let $S$ be an inverse semigroup with congruences $\alpha,\beta$. Then $\alpha$ centralizes $\beta$ if and only if $\TC(m,\alpha,\beta)$.
\end{prp}
\begin{proof}
The necessity is clear from the definitions. For the sufficiency, we first prove $\TC(m,\alpha,\beta)$ implies both $\TC(\ell,\alpha,\beta)$ and $\TC(r,\alpha,\beta)$.
Assume that $a\,\alpha\,b$, $u\,\beta\,v$, and $\ell(a,u) = \ell(a,v)$, that is, $au=av$. Then
\[
m(a,bb^{-1},u) = bb\inv au = bb\inv av = m(a,bb^{-1},v)\,.
\]
Applying $\TC(m,\alpha,\beta)$, we have
\[
bu = bb\inv bu = m(b,bb^{-1},u) = m(b,bb^{-1},v) = bb\inv bv = bv\,.
\]
Thus $\TC(\ell,\alpha,\beta)$ and a dual argument yields $\TC(r,\alpha,\beta)$.

Consider a term $t(x_0,x_1,\dots,x_n) = x_{i_1}^{\epsilon_1}x_{i_2}^{\epsilon_2}\cdots x_{i_k}^{\epsilon_k}$ in normal form with a single occurrence of $x_0$.
Again assume $a\,\alpha\,b$ and $u_i\,\beta\,v_i$ for all $i$. Then $u_{i_p}^{\epsilon_p}\cdots u_{i_q}^{\epsilon_q}\,\beta\, v_{i_p}^{\epsilon_p}\cdots v_{i_q}^{\epsilon_q}$

First, assume that $x_0$ occurs with a positive exponent. There are several cases:
\begin{itemize}
\item $t=x_0$. Then $\TC(t,\alpha,\beta)$ is always true.
\item $x_0$ is the leftmost variable of $t$. Then $\TC(t,\alpha,\beta)$ follows from applying
$\TC(\ell,\alpha,\beta)$ to $\ell(a,u_{i_2}^{\epsilon_2}\cdots u_{i_k}^{\epsilon_k}) = \ell(a,v_{i_2}^{\epsilon_2}\cdots v_{i_k}^{\epsilon_k})$.
\item $x_0$ is the rightmost variable of $t$. A dual argument shows that $\TC(t,\alpha,\beta)$ follows from $\TC(r,\alpha,\beta)$.
\item $x_0$ is in the $j$th place, $1<j<k$. Then $\TC(t,\alpha,\beta)$ follows from applying $\TC(m,\alpha,\beta)$ to
\[
m(a,u_{i_1}^{\epsilon_1}\cdots u_{i_{j-1}}^{\epsilon_{j-1}},u_{i_{j+1}}^{\epsilon_{j+1}}\cdots u_{i_k}^{\epsilon_k})=m(a,v_{i_1}^{\epsilon_1}\cdots v_{i_{j-1}}^{\epsilon_{j-1}},v_{i_{j+1}}^{\epsilon_{j+1}}\cdots v_{i_k}^{\epsilon_k})\,.
\]
\end{itemize}
Finally, assume that $x_0$ occurs with a negative exponent. Since $x\mapsto x\inv$ is a permutation of $S$, and congruences are invariant with respect to inversion, $\TC(t,\alpha,\beta)$ follows from $\TC(t',\alpha,\beta)$ where $t'$ results from $t$ by replacing $x\inv$ for $x$.

This completes the proof.
\end{proof}

\begin{lem}\label{Lem:ab_intersect}
Let $S$ be an inverse semigroup and let $\alpha,\beta$ be congruences of $S$. If $\alpha$ centralizes $\beta$, then $\alpha\cap\beta$ is idempotent-separating.
\end{lem}
\begin{proof}
Assume $e,f\in E(S)$, $e\,\alpha\,f$ and $e\,\beta\,f$. Since $efe = ef$, we may apply $\TC(\ell,\alpha,\beta)$ to $\ell(e,fe) = \ell(e,f)$ to
obtain $\ell(f,fe) = \ell(f,f)$, that is, $ffe = ff$. Thus $fe = f$. Exchanging the roles of $e,f$, we also have $ef=e$. Thus $e=f$ as claimed.
\end{proof}

As in earlier sections, for a normal inverse subsemigroup $N$ of an inverse semigroup $S$, let $\zeta_N = \ker(\Psi_N) = \ker(\Phi_N)$.

\begin{thm}\label{Thm:xi_TC}
Let $S$ be an inverse semigroup, let $\alpha,\beta$ be congruences of $S$, and assume $\alpha\leq \zeta_{N_{\beta}}$. Then $\alpha$ centralizes $\beta$.
\end{thm}
\begin{proof}
  We verify $\TC(m,\alpha,\beta)$ and then the desired result will follow from Proposition \ref{Prp:just_m}. Assume $a\,\alpha\,b$, $x\,\beta\,u$, $y\,\beta\,v$, and
  $xay = uav$. We will prove $xby = ubv$. First,
  \begin{align*}
    xby &= xby(by)\inv by && \\
    &= x\underline{b yy\inv b\inv}by && \\
    &= \underline{xay}y\inv a\inv by && \text{since }(a,b)\in \ker(\Psi_{N_{\beta}}) \\
    &= u\underline{avy\inv a\inv}by && \\
    &= ubvy\inv b\inv by && \text{since }(a,b)\in \ker(\Psi_{N_{\beta}})\text{ and }vy\inv\in N_{\beta} \\
    &= ubv(by)\inv by\,,
  \end{align*}
  that is,
  \begin{equation}\label{Eqn:xi_TC_1}
  xby = ubv(by)\inv by\,.
  \end{equation}
  Next,
  \begin{align*}
    ubv &= ubv(bv)\inv bv && \\
    &= u\underline{bvv\inv b\inv} bv && \\
    &= \underline{uav}v\inv a\inv bv && \text{since }(a,b)\in \ker(\Psi_{N_{\beta}}) \\
    &= x\underline{ayv\inv a\inv} bv && \\
    &= \underline{xby}v\inv b\inv bv &&  \text{since }(a,b)\in \ker(\Psi_{N_{\beta}})\text{ and }yv\inv\in N_{\beta} \\
    &= ubv\underline{(by)\inv by (bv)\inv bv} && \text{by }\eqref{Eqn:xi_TC_1} \\
    &= ubv(bv)\inv bv (by)\inv by && \\
    &= ubv (by)\inv by\,, &&
  \end{align*}
  that is,
  \begin{equation}\label{Eqn:xi_TC_2}
  ubv = ubv(by)\inv by\,.
  \end{equation}
  From \eqref{Eqn:xi_TC_1} and \eqref{Eqn:xi_TC_2}, we have $xby=ubv$, as claimed.
\end{proof}

\subsection{Abelian congruences}\label{ss:abelian}

\begin{lem}\label{Lem:abelian1}
Let $S$ be an inverse semigroup and let $\alpha$ be an abelian congruence. Then $\alpha$ is idempotent-separating and $N_{\alpha}$ is commutative.
\end{lem}
\begin{proof}
The first claim follows from Lemma \ref{Lem:ab_intersect}. For the second, assume $a,b\in N_{\alpha}$ and let $e,f\in E(S)$ be such that
$a\,\alpha\,e$ and $b\,\alpha\,f$. By Lemma \ref{Lem:below H iff tr=0},
$a\greenH e$ and $b\greenH f$. Thus $aa\inv = a\inv a = e$ is the identity element of the $\greenH$-class of $a$ and $bb\inv = b\inv b = f$
is the identity element of the $\greenH$-class of $b$. Since $\alpha$ is idempotent-separating,
every element of $N_{\alpha}$ commutes with every idempotent by Lemma \ref{Lem:Mitsch}. In particular, $a$ and $b$
and their inverses commute with $e$ and $f$.

Since $a$ commutes with $e,f$, we have $a\underline{f}e = e\underline{f}a$. We may thus apply $\TC(m,\alpha,\alpha)$ to $m(f,a,e) = m(f,e,a)$
to get $m(b,a,e) = m(b,e,a)$, that is, $abe = eba$. Since $b$ commutes with $e$ and $ea=ae=a$, we get $ab=ba$, as claimed.
\end{proof}

\begin{lem}\label{Lem:abelian2}
Let $S$ be an inverse semigroup and let $N$ be a commutative, normal inverse subsemigroup. Then the congruence corresponding to the congruence pair $(N,0_{E(S)})$ is abelian.
\end{lem}

\begin{proof}
  First note that $(N,0_{E(S)})$ is a congruence pair by Lemma \ref{Lem:Mitsch}. Let $\alpha$ denote the
  corresponding congruence. We will show that $\alpha\subseteq \zeta_N$ and then apply Theorem \ref{Thm:xi_TC}.

  Assume $g\,\alpha\,h$. Since $\alpha$ is idempotent-separating, $g\inv g = h\inv h$ and $gg\inv = hh\inv$ (Lemma \ref{Lem:below H iff tr=0}), and so
  $gh\inv hg\inv = gg\inv gg\inv = gg\inv = hh\inv$. We use this in the following calculation. For all $a\in N$,
  \begin{align*}
  \underline{g}a\underline{g\inv} &= g\underline{g\inv g}a\underline{g\inv g}g\inv && \\
  &= \underline{gh\inv}\cdot \underline{hah\inv}\cdot hg\inv && \\
  &= hah\inv\cdot \underline{gh\inv\cdot hg\inv} && \text{since }gh\inv, hah\inv \in N \\
  &= hah\inv hh\inv && \\
  &= hah\inv\,.
  \end{align*}
  Thus $\alpha\subseteq \ker(\Psi_N) = \zeta_N$. Therefore $\alpha$ centralizes itself by Theorem \ref{Thm:xi_TC}, that is, $\alpha$ is abelian.
\end{proof}

Theorem \ref{Thm:abelian and central congruence}(1) now follows from Lemmas \ref{Lem:abelian1} and \ref{Lem:abelian2}.

\subsection{The center}

\begin{lem}\label{Lem:center1}
Let $S$ be an inverse semigroup and let $\alpha$ be a central congruence. Then $\alpha\subseteq \zeta_S$.
\end{lem}
\begin{proof}
We use the characterization $\zeta_S = \greenH\cap\, \xi$ where $\xi=\{(a,b):axb=bxa$ for all $x\in S\}$.
The inclusion $\alpha\leq\greenH$ follows from Lemma \ref{Lem:ab_intersect}. For the other inclusion,
let $a\,\alpha\,b$. From the trivial identity $xay\underline{a}z = x\underline{a}yaz$, the term condition $\TC(m,\alpha,1_S)$ implies
  \begin{equation}\label{Eqn:xaybz}
    xaybz = xbyaz
  \end{equation}
  for all $x,y,z\in S$. Thus, for every $x\in S$,
  \begin{align*}
    axb &= aa\inv \underline{a}x\underline{b} b\inv b && \\
     &= aa\inv bxab\inv b && \text{by }\eqref{Eqn:xaybz}  \\
     &= aa\inv bb\inv \underline{b}x\underline{a}a\inv ab\inv b &&  \\
     &= \underline{aa\inv}\underline{bb\inv} axb\underline{a\inv a}\underline{b\inv b} && \text{by }\eqref{Eqn:xaybz} \\
     &= bb\inv aa\inv axb b\inv b a\inv a && \text{since idempotents commute}\\
     &= bb\inv \underline{a}x\underline{b}a\inv a &&\\
     &= bb\inv bxaa\inv a &&  \text{by }\eqref{Eqn:xaybz}\\
     &= bxa\,. && \qedhere
  \end{align*}
\end{proof}

\begin{lem}\label{Lem:center2}
Let $S$ be an inverse semigroup. Then $\zeta_S$ is a central congruence.
\end{lem}
\begin{proof}
Apply Theorem \ref{Thm:xi_TC} to the case $\beta = 1_S$ with $N_{\beta} = S$.
\end{proof}

Theorems \ref{Thm:abelian and central congruence}(2) and the remaining part of \ref{thm:center} follow from Lemmas \ref{Lem:center1} and \ref{Lem:center2}.
In particular, the center $\zeta_S$ is the largest central congruence whose congruence pair is $(Z(S),0_{E(S)})$.

\section{Nilpotent and solvable inverse semigroups}\label{s:nilpotence}

An algebraic structure $A$ is called \emph{abelian} if $\zeta(A) = 1_A$, or, equivalently, if the congruence $1_A$ is abelian.
It is called \emph{nilpotent} (resp. \emph{solvable}) if there is a chain of congruences
\[
0_A = \alpha_0\leq\alpha_1\leq\ldots\leq\alpha_n=1_A
\]
such that $\alpha_{i+1}/\alpha_{i}$ is a central congruence (resp. an abelian congruence) of $A/\alpha_{i}$, for all $i$.
The length of the smallest such series is called the \emph{nilpotency class} (resp. \emph{solvability length}) of $A$.
Alternatively, we can define the \emph{upper central series} inductively by $\zeta_0(A)=0_A$ and $\zeta_{i+1}(A)$ to be the preimage of $\zeta(A/\zeta_i(A))$ and say that $A$ is nilpotent of class $\leq n$ if $\zeta_n(A)=1_A$.

In groups, the notions of abelianness, nilpotence and solvability coincide with the classical terminology. It follows immediately from Theorem \ref{thm:center}(4) that an abelian inverse semigroup is an abelian group, which can also be seen as an application of the main result of \cite{Warne}.

\begin{lem}\label{lm:solvnilp1}
Let $S$ be an inverse semigroup and let $\alpha$ be an idempotent-separating congruence. Then $|E(S/\alpha)|=|E(S)|$.
\end{lem}
\begin{proof}
The congruence $\alpha$ is idempotent-separating, hence $|E(S/\alpha)|\geq|E(S)|$.
On the other hand, if $[a]\in E(S/\alpha)$, then, by Lemma \ref{Lem:homom_images}, there exists $e\in E(S)$ such that $a\,\alpha\,e$, hence $|E(S/\alpha)|\leq|E(S)|$.
\end{proof}

\begin{proof}[Proof of Theorem \ref{Thm:solvnilp}]
Since nilpotence implies solvability, it is sufficient to show that solvable inverse semigroups are groups.
We proceed by induction on $n$, the solvability length of $S$. If $n=1$, then $S$ is an abelian group and the statement holds.
In the induction step, consider the chain $0_S=\alpha_0\leq\alpha_1\leq\ldots\leq\alpha_n=1_S$ witnessing solvability. Then $S/\alpha_1$ is solvable of class $n-1$, hence a group. Since $\alpha_1$ is abelian, it is idempotent-separating by Lemma \ref{Lem:abelian1}, hence $|E(S)|=|E(S/\alpha_1)|=1$ by Lemma \ref{lm:solvnilp1}, and thus $S$ is a group, too.
\end{proof}

\section{Earlier ideas on centrality and nilpotence}\label{s:others}

\subsection{Classical center}

Classically, the term ``center'' in semigroup theory refers to the subsemigroup
\[
C(S) = \{ a\in S\mid ax = xa \text{ for all } x\in S\}\,,
\]
analogous to the use of the word in groups and rings.
Note that $C(S)\subseteq Z(S)$. Indeed, if $a\in C(S)$, then for all $x\in S$,
$aa\inv xa = aa\inv \underline{xa}a\inv a = aa\inv \underline{ax}a\inv a = axa\inv a$.

An inverse semigroup $S$ is a \emph{Clifford semigroup} if $E(S)\subseteq C(S)$,
that is, if idempotents commute with all elements of $S$. A Clifford semigroup $S$ is a semilattice of groups,
$S = \bigcup_{i\in I} S_i$.

\begin{prp}\label{Prp:ZCliff}
  An inverse semigroup $S$ is a Clifford semigroup if and only if $Z(S) = C(S)$.
\end{prp}
\begin{proof}
$(\Leftarrow)$. Since $E(S)\subseteq Z(S)$ in every inverse semigroup, we obtain $E(S)\subseteq C(S)$ and so $S$ is Clifford.
$(\Rightarrow)$. If $S$ is a Clifford semigroup and if $a\in Z(S)$, then for all $x\in S$,
  $ax = aa\inv ax = axa\inv a = aa\inv xa = xaa\inv a = xa$, and so $a\in C(S)$.
\end{proof}

\subsection{Kowol-Mitsch, Meldrum, and Piochi}

Kowol and Mitsch \cite{Kowol_Mitsch} and Meldrum \cite{Meldrum} studied a notion of nilpotence for Clifford semigroups.
Instead of the upper central series of congruences, they considered the ascending series $E(S) = Z_0(S)\leq Z_1(S)\leq \cdots$
of kernel parts.
A Clifford semigroup $S$ is \emph{KMM-nilpotent} if, for some $n\geq 0$, $Z_n(S) = S$; the smallest such $n$ is the \emph{KMM-nilpotency class}
of $S$. (This is not the definition of nilpotence in either \cite{Kowol_Mitsch} or \cite{Meldrum}, but it is equivalent \cite[Cor. 3.9]{Meldrum}.)
A Clifford semigroup $S = \bigcup_{i\in I} S_i$ is KMM-nilpotent of class $n$ if and only if the groups $S_i$ are all nilpotent of class
at most $n$ and at least one of the groups attains that class.

While a nilpotent (in the universal algebra sense) Clifford semigroup is certainly KMM-nilpotent (because it is a group), the converse need
not be true. This is because the condition $Z_n(S) = S$ does not imply $\zeta_n(S) = 1_S$. For instance, a Clifford semigroup $S$
is KMM-nilpotent of class $0$ if and only if it is a semilattice, while $S$ is KMM-nilpotent of class $1$ if and only if it is commutative.

Meldrum also defined a notion of solvability for Clifford semigroups which was later extended to arbitrary inverse semigroups by
Piochi \cite{Piochi}. We will call this \emph{MP-solvability}. We refer to the original papers for the definition. MP-solvability is implied by
universal algebra solvability but the converse is false; for example, an inverse semigroup is MP-solvable of class $1$ if and only if it is commutative.

\subsection{Mal'cev}

For variables $a, b, z_0, z_1,\ldots$, define two sequences of semigroup words by
\begin{align*}
\lambda_0 = a,  \quad &\rho_0 = b \\
\lambda_{n+1} = \lambda_n z_n \rho_n, \quad & \rho_{n+1} = \rho_n z_n \lambda_n
\end{align*}
Then the semigroup $S$ is \emph{Mal'cev nilpotent} \cite{Malcev} if there exists $n$ such that the identity $\lambda_n = \rho_n$ holds for all $a,b,z_0,\ldots,z_{n-1}\in S$. The nilpotency class of $S$ is the smallest $n$ satisfying this condition. Mal'cev proved that in groups, his nilpotence is the same as the usual nilpotence, and thus group nilpotence can be entirely described in terms of semigroup identities.

Mal'cev nilpotence and its variants have been extensively studied in the semigroup literature \cite{AS22,ASK,JO}. For simplicity, we are following Mal'cev's original definition, but some authors follow Lallement \cite{Lallement} and assume the variables $z_0,z_1,\ldots$ are allowed to take values in $S^1$, giving a slightly stronger notion of nilpotence.

Meldrum \cite{Meldrum} observed that for $n > 0$, a Clifford semigroup $S = \bigcup_{i\in I}S_i$ is Mal'cev nilpotent of class at most $n$ if and only if each group $S_i$ is nilpotent of class at most $n$, and thus KMM-nilpotence and Mal'cev nilpotence coincide for Clifford semigroups \cite[p. 10]{Meldrum}. (The two notions differ at class $0$.)

According to Theorem \ref{Thm:solvnilp}, inverse semigroups which are nilpotent in the universal algebra sense are groups. By contrast, the
Brandt semigroup of order $5$ is Mal'cev nilpotent of class $2$. Thus the two notions of nilpotence do not coincide.

Now suppose we view the equation $\lambda_n = \rho_n$ as defining a binary relation on $S$:
\[
a\, \mu_n\, b \text{ iff } \lambda_n = \rho_n \text{ for all }z_0, ..., z_{n-1}\in S\,.
\]
In inverse semigroups, these relations turn out to be tolerances: reflexive, symmetric and compatible (i.e., inverse subsemigroups of $S\times S$). We will call these the \emph{Mal'cev tolerances} of $S$. Here are the first few written explicitly:
\begin{align*}
a\, \mu_0\, b  &\text{ iff }  a = b, \\
a\, \mu_1\, b  &\text{ iff }  a z_0 b = b z_0 a \text{ for all }z_0\in S\,,  \\
a\, \mu_2\, b  &\text{ iff }  a z_0 b z_1 b z_0 a = b z_0 a z_1 a z_0 b \text{ for all }z_0,z_1\in S\,.
\end{align*}
By Theorem \ref{thm:center}, $\zeta_1(S) = \greenH\cap\,\mu_1$. Using Prover9 \cite{McCune}, we have verified that $\zeta_2(S) = \greenH\cap\,\mu_2$.

\begin{prb}
Let $S$ be an inverse semigroup with upper central series $0\leq \zeta_1(S)\leq \zeta_2(S)\leq \cdots$ and Mal'cev tolerances
$\mu_0, \mu_1,\ldots$. Does $\zeta_n(S) = \mathcal{H}\cap\,\mu_n$ hold for all $n\geq 0$?
\end{prb}

\end{document}